\documentclass{amsart}

\usepackage{amsmath, amsrefs, amssymb, amsthm, enumerate, thmtools}
\usepackage[hidelinks]{hyperref}
\usepackage[all]{xy}

\declaretheorem[name=Theorem, sibling=equation]{theorem}
\declaretheorem[name=Lemma, sibling=equation]{lemma}
\declaretheorem[name=Definition, style=definition, sibling=equation]{definition}
\declaretheorem[name=Notation, style=definition, sibling=equation]{notation}
\declaretheorem[name=Example, style=definition, sibling=equation]{example}

\newcommand{\N}{{\mathbb{N}}}

\title{A noetherian criterion for sequences of modules}

\author{Wee Liang Gan}
\address{}
\email{}

\author{Khoa Ta}
\address{}
\email{}

\begin{document}

\begin{abstract}
   We prove a noetherian criterion for a sequence of modules with linear maps between them. This generalizes a noetherian criterion of Gan and Li for infinite EI categories. We apply our criterion to the linear categories associated to certain diagram algebras defined by Patzt.
\end{abstract}

\maketitle

Gan and Li \cite{gl} proved a noetherian criterion for infinite EI categories. In this note, we adapt their proof to a more general setting and apply it to the linear categories studied by Patzt \cite{patzt} in the context of representation stability. In contrast to \cite{gl} and \cite{ss}, there is no combinatorial category in our setup. 

We work over a field $k$. Denote by $\N$ the set of nonnegative integers. Let $A$ be a sequence $\{A_i\}_{i\in \N}$ where $A_i$ is a $k$-algebra for each $i\in \N$.

\begin{definition}
   An \emph{$A$-module} $M$ is a sequence of pairs $\{(M_i, \phi^M_i)\}_{i\in \N}$ where, for each $i\in \N$,
   \begin{itemize}
      \item $M_i$ is an $A_i$-module;
      \item $\phi^M_i: M_i \to M_{i+1}$ is a $k$-linear map.
   \end{itemize}
\end{definition}

\begin{definition}
   Let $M$ and $N$ be $A$-modules. A \emph{morphism} $f: M\to N$ of $A$-modules is a sequence $\{f_i\}_{i\in \N}$ where, for each $i\in \N$, 
   \begin{itemize}
      \item $f_i: M_i\to N_i$ is a homomorphism of $A_i$-modules;
      \item $\phi^N_i \circ f_i = f_{i+1} \circ \phi^M_i$.
   \end{itemize}
\end{definition}

\begin{definition}
   Let $M$ be an $A$-module.
 \begin{enumerate}[(i)]
   \item An $A$-module $N$ is an \emph{$A$-submodule} of $M$ if each $N_i$ is a subset of $M_i$ and the sequence of inclusion maps $\{N_i\hookrightarrow M_i\}_{i\in \N}$ is a morphism of $A$-modules (called the \emph{inclusion morphism}).  
   
   \item Let $E\subseteq \bigsqcup_{i\in \N} M_i$. The $A$-submodule of $M$ \emph{generated by} $E$ is the smallest $A$-submodule $N$ of $M$ such that $E\subseteq \bigsqcup_{i\in \N} N_i$. 
 \end{enumerate}
\end{definition}

\begin{definition}
   \begin{enumerate}[(i)] 
      \item An $A$-module is \emph{finitely generated} if it can be generated by a finite set.

      \item An $A$-module is \emph{noetherian} if each of its $A$-submodules is finitely generated.
   \end{enumerate}
\end{definition}

\begin{notation}
   \begin{enumerate}[(i)]
      \item Denote by $\mathrm{Mod}(k)$ the category of $k$-modules. 
      \item Let $M$ be an $A$-module. Denote by $\mathrm{Sub}_A(M)$ the category whose objects are the $A$-submodules of $M$ and whose morphisms are inclusion morphisms between the $A$-submodules of $M$.
      For each $i\in \N$, denote by $F^M_i$ the functor
         \[ F^M_i : \mathrm{Sub}_A(M) \to \mathrm{Mod}(k), \quad N\mapsto \mathrm{Hom}_{A_i}(M_i, N_i). \] 
      (An inclusion morphism $N' \hookrightarrow N$ between $A$-submodules of $M$ induces an inclusion map $F^M_i(N') \hookrightarrow F^M_i(N)$.) In particular, 
      \[ F^M_i(M) = \mathrm{End}_{A_i}(M_i). \]
   \end{enumerate}
\end{notation}

\begin{lemma} \label{lemma.a}
   Let $M$ be an $A$-module, let $N$ be an $A$-submodule of $M$, and let $N'$ be an $A$-submodule of $N$. Let $i\in \N$ and assume that:
   \begin{itemize}
      \item $M_i$ is a semisimple $A_i$-module;
      \item $N'_i \subsetneqq N_i$.
   \end{itemize}
   Then $F^M_i(N') \subsetneqq F^M_i(N)$.
\end{lemma}

\begin{proof}
   Since $M_i$ is a semisimple $A_i$-module, there exists an $A_i$-submodule $Q$ of $M_i$ such that $M_i=N_i \oplus Q$.
   Let $h$ be the projection map from $M_i$ onto $N_i$ along $Q$. Then $h$ belongs to $F^M_i(N)$. Since $N'_i \neq N_i$, we know that $h$ does not factor through the inclusion map $N'_i\hookrightarrow N_i$, hence $h$ does not belong to $F^M_i(N')$.
\end{proof}

Our noetherian criterion is:

\begin{theorem} \label{thm.a}
   Let $M$ be an $A$-module such that each $M_i$ is finite dimensional. Let $d\in \N$. Assume that for each integer $i\geqslant d$,
   \begin{enumerate}
      \item $M_i$ is a semisimple $A_i$-module;
      \item there exists a morphism of functors $\nu_i: F^M_i \to F^M_{i+1}$ such that 
         \[ \nu_i(M): F^M_i (M) \to F^M_{i+1}(M)\]  
      is a bijection. 
   \end{enumerate} 
   Then $M$ is noetherian.
\end{theorem}

\begin{proof}
   We adapt the proof of \cite{gl}*{Proposition 4.10} to our setting. 
   
   Assume on the contrary that there exists an $A$-submodule $N$ of $M$ which is not finitely generated. For each $\ell \in \N$, let $N^{(\ell)}$ be the $A$-submodule of $N$ generated by $\bigsqcup_{i=0}^\ell N_i$. Observe that if $i\in\{0,\ldots, \ell\}$, then $N^{(\ell)}_i = N_i$.
   
   Since each $N_i$ is finite dimensional, the $A$-module $N^{(\ell)}$ is finitely generated; but $N$ is not finitely generated, so $N^{(\ell)}\subsetneqq N$. It follows that there exists an integer $d_\ell > \ell$ such that $N^{(\ell)}_{d_\ell} \subsetneqq N_{d_\ell}$. We have $N^{(\ell)}_{d_\ell} \subsetneqq N^{(d_\ell)}_{d_\ell}$.

   Define a sequence of integers $\{\ell_i\}_{i\in \N}$ recursively by $\ell_0 = d$ and
   \[ \ell_{i+1} = d_{\ell_i} \quad \mbox{ for each } i\in \N. \]
   We have $d=\ell_0 < \ell_1 < \ell_2 < \cdots$. For each $i\in \N$,
   \[ N^{(\ell_i)}_{\ell_{i+1}} \subsetneqq N^{(\ell_{i+1})}_{\ell_{i+1}}. \]
   Hence, by Lemma \ref{lemma.a} and (\ref{thm.a}.1),
   \[ F^M_{\ell_{i+1}} (N^{(\ell_i)}) \subsetneqq F^M_{\ell_{i+1}} (N^{(\ell_{i+1})}). \]
   We know that $F^M_{\ell_{i+1}} (N^{(\ell_{i+1})})$ is finite dimensional because $M_{\ell_{i+1}}$ and $N^{(\ell_{i+1})}_{\ell_{i+1}}$ are finite dimensional. It follows that
   \begin{equation} \label{eqn.a}
      \dim F^M_{\ell_{i+1}} (N^{(\ell_i)}) < \dim F^M_{\ell_{i+1}} (N^{(\ell_{i+1})}). 
   \end{equation}

   By (\ref{thm.a}.2), we also have, for each $i\in \N$, a commuting diagram
   \[ \xymatrixcolsep{2cm}\xymatrix{  
      F^M_d(N^{(\ell_i)}) \ar[r]^-{\nu_d(N^{(\ell_i)})} \ar@{^{(}->}[d] & F^M_{d+1}(N^{(\ell_i)}) \ar[r]^-{\nu_{d+1}(N^{(\ell_i)})}  \ar@{^{(}->}[d] & F^M_{d+2}(N^{(\ell_i)}) \ar[r]^-{\nu_{d+2}(N^{(\ell_i)})} \ar@{^{(}->}[d] & \cdots \\
      F^M_d(M) \ar[r]^-{\nu_d(M)} & F^M_{d+1}(M) \ar[r]^-{\nu_{d+1}(M)} & F^M_{d+2}(M) \ar[r]^-{\nu_{d+2}(M)} & \cdots \\
     } \]   
   In the above diagram, the maps in the second row are bijective, hence the maps in the first row are injective. We deduce that, for each $i\in \N$,
   \begin{equation} \label{eqn.b}
      \dim F^M_{\ell_i}(N^{(\ell_i)}) \leqslant \dim F^M_{\ell_{i+1}}(N^{(\ell_i)})
   \end{equation}
   and 
   \begin{equation} \label{eqn.c}
      \dim F^M_{\ell_i}(N^{(\ell_i)}) \leqslant \dim F^M_d (M).
   \end{equation}

   From (\ref{eqn.a}) and (\ref{eqn.b}), we obtain, for each $i\in \N$,
   \[ \dim F^M_{\ell_i}(N^{(\ell_i)}) < \dim F^M_{\ell_{i+1}} (N^{(\ell_{i+1})}). \]
   Hence,
   \[ \dim F^M_{\ell_0}(N^{(\ell_0)}) < \dim F^M_{\ell_1}(N^{(\ell_1)}) < \dim F^M_{\ell_2}(N^{(\ell_2)}) < \cdots. \]
   This is a contradiction because by (\ref{eqn.c}) these dimensions are at most $\dim F^M_d(M)$, but $\dim M_d < \infty$ and $F^M_d(M)=\mathrm{End}_{A_d}(M_d)$ imply that $\dim F^M_d(M)<\infty$.
\end{proof}

\begin{example} \label{example.a}
   Let $k=\mathbb{C}$. 
   Let $\{A_i\}_{i\in\N}$ be one of the following three sequences of $\mathbb{C}$-algebras associated to a parameter $\delta\in \mathbb{C}$ (see \cite{patzt} for definitions):
   \begin{itemize}
      \item the sequence $\{\mathrm{TL}_i\}_{i\in\N}$ of Temperley-Lieb algebras,
      \item the sequence $\{\mathrm{Br}_i\}_{i\in\N}$ of Brauer algebras,
      \item the sequence $\{\mathrm{P}_i\}_{i\in\N}$ of partition algebras.
   \end{itemize}   
   Assume that $\delta$ is chosen so that each $A_i$ is a semisimple $\mathbb{C}$-algebra. Let $A$ be the sequence $\{A_i\}_{i\in\N}$.

   Let $C_A$ be the linear category (called stability category) defined in \cite{patzt}*{page 635}. The set of objects of $C_A$ is $\N$. For any $i,j\in \N$, if $i\leqslant j$, then 
   \[ C_A(i,j)= A_j \otimes_{A_{j-i}}\mathbb{C}; \]
   otherwise if $i>j$, then $C_A(i,j)=0$. For each $m\in\N$, the functor $C_A(m, -)$ is a $C_A$-module denoted by $M(m)$ (see \cite{patzt}*{page 643}.
   
   We claim that every finitely generated $C_A$-module is noetherian\footnote{It follows that Theorems A, B, C of \cite{patzt} hold for any finitely generated $C_A$-module. 
   }. Let $m\in \N$. It suffices\footnote{Every finitely generated $C_A$-module $V$ is a homomorphic image of $M(m_1)\oplus \cdots\oplus M(m_r)$ for some $m_1, \ldots, m_r\in \N$. If $M(m_1), \ldots, M(m_r)$ are noetherian, then their direct sum is noetherian, which implies that $V$ is noetherian.} to prove that the $C_A$-module $M(m)$ is noetherian. To this end, define an $A$-module $M$ as follows: for each $i\in\N$, let $M_i = M(m)_i$ and let $\phi^M_i : M_i \to M_{i+1}$ be the map induced by the element $1\otimes 1\in A_{i+1}\otimes_{A_1} \mathbb{C}$ (see \cite{patzt}*{Lemma 2.14}). Then each $M_i$ is a finite dimensional semisimple $A_i$-module. To apply Theorem \ref{thm.a}, it remains to verify that there exists $d\in\N$ such that for each $i\geqslant d$, condition (\ref{thm.a}.2) holds.

   For each integer $i\geqslant m$, define a functor $F'_i$ by 
   \[ F'_i : \mathrm{Sub}_A(M) \to \mathrm{Mod}(\mathbb{C}), \quad N \mapsto \mathbb{C}\otimes_{A_{i-m}} N_i. \]

   \medskip 
   
   \noindent{\it Claim \ref{example.a}.1.} For each integer $i\geqslant m$, the functors $F^M_i$ and $F'_i$ are isomorphic.
   
   \begin{proof}
   We have:
   \begin{align*}
      F^M_i (N) &= \mathrm{Hom}_{A_i} (A_i\otimes_{A_{i-m}} \mathbb{C}, N_i) \\
      &\cong \mathrm{Hom}_{A_{i-m}} (\mathbb{C}, N_i).
   \end{align*}
   Since $N_i$ is a semisimple $A_{i-m}$-module, it is a direct sum of its isotypic components. Let $N_i^{\mathrm{triv}}$ be the isotypic component of $N_i$ spanned by the submodules isomorphic to the trivial $A_{i-m}$-module $\mathbb{C}$.  Then we have:
   \begin{align*}
      F^M_i(N) &\cong N_i^{\mathrm{triv}} \\
      &\cong \mathbb{C}\otimes_{A_{i-m}}  N_i \quad \mbox{(using \cite{patzt}*{Lemma 4.5})}\\
      &= F'_i(N). 
   \end{align*}
   \end{proof} 

   For each integer $i\geqslant m$ and $A$-submodule $N$ of $M$, the map $\phi^N_i: N_i \to N_{i+1}$ induces a map 
   \[ \nu'_i(N) : \mathbb{C}\otimes_{A_{i-m}} N_i \to \mathbb{C}\otimes_{A_{i+1-m}} N_{i+1}. \]
   This defines a functor $\nu'_i : F'_i \to F'_{i+1}$. By \cite{patzt}*{Theorem 4.3}, for all $i$ sufficiently large, the map 
   \[ \nu'_i(M): F'_i(M) \to F'_{i+1}(M) \]  
   is a bijection and hence by Claim \ref{example.a}.1, condition (\ref{thm.a}.2) holds. We conclude by Theorem \ref{thm.a} that the $A$-module $M$ (and hence the $C_A$-module $M(m)$) is noetherian. 
\end{example}

\begin{example} \label{example.b}
   We show how to apply Theorem \ref{thm.a} to the setting in \cite{gl}. Let $k$ be a field of characteristic $0$ and $C$ the category in \cite{gl}*{Theorem 3.7}. 

   By assumption, $C$ is a category whose set of objects is $\N$, and for all $i,j, \ell\in \N$:
   \begin{itemize}
      \item $C(i,j)$ is a finite set;
      \item $C(i,j)$ is nonempty if and only if $i\leqslant j$;
      \item if $i<j<\ell$, then the composition map 
      \[ C(i,j)\times C(j,\ell) \to C(i,\ell), \quad (\alpha, \beta) \mapsto \beta\alpha \]
      is surjective;
      \item every morphism in $C(i,i)$ is an isomorphism, so that $C(i,i)$ is a group which we denote by $G_i$;
      \item (\emph{transitivity condition}) the action of $G_{i+1}$ on $C(i,i+1)$ is transitive.
   \end{itemize}
   Moreover, $C$ is assumed to satisfy a bijectivity condition. To state this condition, we need some notations. 

   For each $i\in \N$, fix a morphism $\alpha_i\in C(i,i+1)$. Then whenever $i\leqslant j$, we have the morphism $\alpha_{j-1}\cdots\alpha_{i+1}\alpha_i\in C(i,j)$. Let $H_{i,j}$ be the stabilizer of $\alpha_{j-1}\cdots\alpha_{i+1}\alpha_i$ in the group $G_j$. (In particular, $H_{i,i}$ is the trivial subgroup of $G_i$.)
   
   \medskip

   \noindent{\it Claim \ref{example.b}.1.} For each $h\in H_{i,j}$, there exists $g\in H_{i,j+1}$ such that $\alpha_j h = g\alpha_j$.

   \begin{proof} 
      By the transitivity condition, for each $h\in H_{i,j}$, there exists $g\in G_{j+1}$ such that $\alpha_j h = g\alpha_j$. Observe that 
      \[ g\alpha_j \alpha_{j-1}\cdots \alpha_i = \alpha_j h \alpha_{j-1} \cdots \alpha_i =  \alpha_j \alpha_{j-1} \cdots \alpha_i. \]
      Thus $g\in H_{i,j+1}$.
   \end{proof}

   Denote by $H_{i,j}\backslash C(i,j)$ the set of $H_{i,j}$-orbits in $C(i,j)$. It follows from Claim \ref{example.b}.1 that the map 
   \[ C(i,j) \to C(i, j+1), \quad \beta \mapsto \alpha_j \beta \]
   sends each $H_{i,j}$-orbit into a $H_{i,j+1}$-orbit. Let
   \[ \mu_{i,j}: H_{i,j}\backslash C(i,j) \to H_{i,j+1}\backslash C(i,j+1) \]
   be the map such that for each $\beta\in C(i,j)$, if $O$ is the $H_{i,j}$-orbit containing $\beta$, then $\mu_{i,j}(O)$ is the $H_{i,j+1}$-orbit containing $\alpha_j\beta$. The \emph{bijectivity condition} states that: for each $i\in \N$, the map $\mu_{i,j}$ is bijective for all $j$ sufficiently large.

   Now for each $i\in\N$, let $A_i$ be the group algebra $kG_i$. Let $A$ be the sequence $\{A_i\}_{i\in \N}$.
   
   Denote by $kC(i,j)$ the vector space over $k$ with basis $C(i,j)$. For each $m\in\N$, the functor $kC(m, -)$ is a $C$-module over $k$ denoted by $M(m)$. We wish to deduce that $M(m)$ is a noetherian $C$-module from Theorem \ref{thm.a}. To this end, define an $A$-module $M$ as follows: for each $i\in\N$, let $M_i = M(m)_i$ and let $\phi^M_i : M_i \to M_{i+1}$ be the map induced by $\alpha_i$. 
   
   By Maschke's theorem, we know that $M_i$ is a semisimple $A_i$-module. We need to show that there exists $d\in\N$ such that for each $i\geqslant d$, condition (\ref{thm.a}.2) holds.

   For each integer $i\geqslant m$, define a functor $F'_i$ by 
   \[ F'_i : \mathrm{Sub}_A(M) \to \mathrm{Mod}(k), \quad N \mapsto k\otimes_{H_{m,i}} N_i. \]
   Similarly to Claim \ref{example.a}.1, the functors $F^M_i$ and $F'_i$ are isomorphic. 
   
   Now consider any $A$-submodule $N$ of $M$. Using Claim \ref{example.b}.1, it is easy to see that the map $\phi^N_i: N_i \to N_{i+1}$ induces a map 
   \[ \nu'_i(N) : k\otimes_{H_{m,i}} N_i \to k\otimes_{H_{m,i+1}} N_{i+1}. \]
   Thus we obtain a functor $\nu'_i : F'_i \to F'_{i+1}$. We have 
   \begin{align*}
      F'_i(M) &= k\otimes_{H_{m,i}} M_i \\
      &= k\otimes_{H_{m,i}} kC(m,i).
   \end{align*}
   It follows from the bijectivity condition on $C$ that for all $i$ sufficiently large, the map 
   \[ \nu'_i(M): F'_i(M) \to F'_{i+1}(M) \]  
   is a bijection and hence condition (\ref{thm.a}.2) holds. 
\end{example}

\end{document}